\documentclass[12pt]{amsart}

\usepackage{tgpagella}
\usepackage{euler}
\usepackage[T1]{fontenc}
\usepackage{amsmath, amssymb}
\usepackage[hidelinks]{hyperref}
\usepackage[english]{babel}
\usepackage{mathrsfs}
\usepackage{eucal}
\usepackage[all]{xy}
\usepackage{tikz}

\newtheorem{thm}{Theorem}[section]
\newtheorem*{thm*}{Theorem}
\newtheorem{lem}[thm]{Lemma}
\newtheorem{fact}[thm]{Fact}

\newtheorem{prop}[thm]{Proposition}
\newtheorem*{prop*}{Proposition}

\newtheorem{cor}[thm]{Corollary}
\newtheorem*{cor*}{Corollary}

\theoremstyle{definition}
\newtheorem{defn}[thm]{Definition}
\newtheorem*{defn*}{Definition}

\newtheorem{remark}[thm]{Remark}

\newtheorem*{question*}{Question}
\newtheorem*{Pquestion*}{Popa's question}

\newtheorem*{conv*}{Convention}

\newcommand{\dminus}{ 
\buildrel\textstyle\ .\over{\hbox{ 
\vrule height3pt depth0pt width0pt}{\smash-} 
}}

\def\bb{\mathbb}

\def\bb{\mathbb}

\def\cal{\mathcal}

\def\u{\mathsf 1}

\newcommand{\cstar}{$\mathrm{C}^*$}

\makeatletter

\def\dotminussym#1#2{%
  \setbox0=\hbox{$\m@th#1-$}%
  \kern.5\wd0%
  \hbox to 0pt{\hss\hbox{$\m@th#1-$}\hss}%
  \raise.6\ht0\hbox to 0pt{\hss$\m@th#1.$\hss}%
  \kern.5\wd0}
\newcommand{\dotminus}{\mathbin{\mathpalette\dotminussym{}}}

\DeclareMathOperator{\tr}{tr}

\def \Th{\operatorname{Th}}
\def \R{\mathcal R}
\def \u{\mathcal U}
\def \val{\operatorname{val}}
\def \pval{\operatorname{s-val}}
\newcommand{\mip}{\operatorname{MIP}}

\newcommand{\cqs}{C_q}
\newcommand{\cqa}{C_{qa}}
\newcommand{\cqc}{C_{qc}}

\begin{document}


\title{The universal theory of the hyperfinite II$_1$ factor is not computable}
\author{Isaac Goldbring and Bradd Hart}

\address{Department of Mathematics\\University of California, Irvine, 340 Rowland Hall (Bldg.\# 400),
Irvine, CA 92697-3875}
\email{isaac@math.uci.edu}
\urladdr{http://www.math.uci.edu/~isaac}

\address{Department of Mathematics and Statistics, McMaster University, 1280 Main St., Hamilton ON, Canada L8S 4K1}
\email{hartb@mcmaster.ca}
\urladdr{http://ms.mcmaster.ca/~bradd/}

\begin{abstract}
    We show that the universal theory of the hyperfinite II$_1$ factor is not computable.  The proof uses the recent result that MIP*=RE.  Combined with an earlier observation of the authors, this yields a proof that the Connes Embedding Problem has a negative solution that avoids the equivalences with Kirchberg's QWEP Conjecture and Tsirelson's Problem.
\end{abstract}

\maketitle

\section{Introduction}

In this note, $\cal R$ denotes the hyperfinite II$_1$ factor.  The main result of this note is the following result:

\begin{thm*}
The universal theory of $\R$ is not computable.
\end{thm*}

In the next section, we will define this statement precisely.  Roughly speaking, this says that there is no algorithm which takes as inputs a universal sentence and rational tolerance $\epsilon>0$ and produces an interval of radius less than $\epsilon$ containing the truth value of the sentence in $\cal R$.

In Section 4, we offer an alternative formulation which shows that our main result is really a result about matrices and traces.  Given positive integers $n$ and $d$, we fix variables $x_1,\ldots,x_n$ and enumerate all *-monomials in the variables $x_1,\ldots,x_n$ of total degree at most $d$, $m_1,\ldots,m_L$. (Of course, $L=L(n,d)$ depends on both $n$ and $d$.)  We consider the map $\mu_{n,d}:\R_1^n \rightarrow D^L$ given by $\mu_{n,d}(\vec a)=(\tau(m_i(\vec a)) \ : \ i=1,\ldots,L)$.  (Here, $D$ is the complex unit disk.)


We let $X(n,d)$ denote the range of $\mu_{n,d}$ and $X(n,d,p)$ be the image of the unit ball of $M_p(\mathbb C)$ under $\mu_{n,d}$.  Notice that $\bigcup_{p\in \bb N} X(n,d,p)$ is dense in $X(n,d)$.  Thus, there is a function $F:\mathbb N^3\to \mathbb N$ such that $X(n,d,F(n,d,k))$ is $\frac{1}{k}$-dense in $X(n,d)$ for all $n,d,k\in \mathbb N$.  The fact that the universal theory of $\R$ is not computable is equivalent to the fact that no such $F$ is a computable function.

The proof of our main theorem appears in Section 3 and uses the recent result from \cite{MIP*} that shows that the complexity classes MIP* and RE are the same.  As noted in \cite{MIP*}, that result can be used to show that the \textbf{Connes Embedding Problem} (CEP) has a negative solution.  Recall that CEP asks whether or not every II$_1$ factor embeds into an ultrapower of $\cal R$.  The argument presented in \cite{MIP*} that MIP*=RE implies the failure of CEP is complicated.  First, one shows that MIP*=RE implies that \textbf{Tsirelson's Problem} has a negative solution; this fact first appears in \cite{FNT}.  Next, one uses that the failure of Tsirelson's Problem implies that Kirchberg's \textbf{QWEP Conjecture} has a negative solution; this fact is due to Fritz and Junge et al \cite{Fr,Junge}.  (That Tsirelson's Problem is actually equivalent to the QWEP conjecture is due to Ozawa \cite{Oz}.)  Finally, one uses that the failure of the QWEP Conjecture implies the failure of CEP, which appears in \cite{K}.

The current authors showed in \cite{GH} that a positive solution to CEP implies that the universal theory of $\cal R$ is computable.  The proof is essentially an immediate application of the \textbf{Completeness Theorem} for continuous first order logic \cite{BYP} and the fact that the theory of II$_1$ factors has a recursively enumerable axiomatization.  Thus, the main theorem here yields a proof that MIP*=RE implies that CEP has a negative solution using just basic facts from continuous logic.

In Section 5, we offer a general perspective on embedding problems and point out how our techniques give a stronger refutation of the CEP in the spirit of the G\"odel Incompleteness Theorem.  In particular, we prove the following result:

\begin{thm*}
If $T$ is any consistent, recursively axiomatizable extension of the theory of II$_1$ factors, then there is a II$_1$ factor which satisfies $T$ which does not embed into an ultrapower of $\R$.
\end{thm*}

This theorem allows us to prove that therer are ``many'' counterexamples to CEP in a sense we now make precise.  In \cite{mtoa3}, the authors proved the existence of so-called \textbf{locally univeral} II$_1$ factors, that is, separable II$_1$ factors $M$ such that every separable II$_1$ factor embeds into an ultrapower of $M$.  The negative solution to CEP provided by MIP*$=$RE thus asserts that $\cal R$ is not a locally universal II$_1$ factor.  It is a priori possible that all II$_1$ factors fall into one of two categories:  those that are embeddable into an ultrapower of $\cal R$ and those that are locally universal.  Our previous theorem can be used to show that this is emphatically not the case:

\begin{thm*}
There is a sequence $M_1,M_2,\ldots,$ of separable II$_1$ factors, none of which embed into an ultrapower of $\cal R$, and such that, for all $i<j$, $M_i$ does not embed into an ultrapower of $M_j$.
\end{thm*}

In Section 6, we present some applications of our techniques to a large class of C*-algebras.  When applied to the universal UHF algebra $\mathcal{Q}$, we obtain a G\"odelian-style refutation of the MF problem, which asks whether every stably finite \cstar-algebra embeds into an ultrapower of $\mathcal{Q}$.  On the other hand, when applied to the case of the Jiang-Su algebra $\mathcal{Z}$, we obtain the following purely operator-algebraic consequence, which appears to be new:

\begin{thm*}
There is a stably projectionless \cstar-algebra that does not embed into an ultrapower of the Jiang-Su algebra $\mathcal{Z}$.
\end{thm*}

In the final section, we offer alternative proofs to the negative solutions of Tsirelson's problem and Kirchberg's QWEP conjecture from MIP*=RE, replacing the semidefinite programming argument from \cite{FNT} with a simple application of the Completeness Theorem.


In order to keep this note short, we include very little background information on continuous logic (the material that is truly necessary for our proof appears in the next section) or quantum games.  We refer the reader to \cite{mtfms} for continuous logic or \cite{MTC*A} for an operator algebraic approach; the introduction to \cite{MIP*} contains an excellent guide to the necessary work on quantum games.  A first version of the proof of the main theorem was given in a talk at the Canadian Operator Symposium in May, 2020.  We would like to thank Se-Jin Kim, Vern Paulsen and Chris Schafhauser for pointing out the simplification possible by considering synchronous correlation sets and a special thanks to Thomas Vidick for providing the additional information regarding the role of such correlation sets in the proof in \cite{MIP*}.  We would also like to thank Thomas Sinclair, Aaron Tikuisis, Mikael R\o rdam and Jamie Gabe for enlightening discussions around the MF and other embedding problems, and to Ward Henson for useful comments about the computability-theoretic issues under discussion. 

\section{A little continuous logic}

We fix a countable collection $(u_n)$ of continuous functions $\mathbb R^k\to \mathbb R$ (as $k$ varies) with compact support satisfying the following two conditions:
\begin{enumerate}
    \item For each $k$, the set $\{u_n\ : \ n\in \bb N\}\cap C_c(\bb R^k)$ is dense in $C_c(\bb R^k)$.
    \item There is an algorithm that takes as inputs a computable $f\in C_c(\bb R^k)$ and a rational $\delta>0$ and returns $n$ such that $u_n\in C(\bb R^k)$ and $\|f-u_n\|_\infty<\delta$.
\end{enumerate}

For convenience, we assume that the following functions are amongst the sequence $(u_n)$:
\begin{itemize}
    \item the binary functions $+$ and $\cdot$,
    \item for each $\lambda\in \bb Q$, the unary function $x\mapsto \lambda x$,
    \item the binary function $\dminus$ given by $x\dminus y:=\max(x-y,0)$, and 
    \item the unary functions $x\mapsto 0$, $x\mapsto 1$, and $x\mapsto \frac{x}{2}$.
\end{itemize}

We now fix a computable continuous language $L$.  (In the next section, $L$ will be the langauge of tracial von Neumann algebras.)  We call an $L$-formula \textbf{restricted} if it only uses functions appearing in the sequence $u_n$ as connectives.  We fix an enumeration $(\varphi_m)$ of the restricted $L$-formulae.  We also call an $L$-formula $\textbf{computable}$ if it only uses computable connectives. The following is immediate from the definitions:

\begin{lem}
There is an algorithm such that takes as inputs a computable $L$-formula $\varphi(x)$ and rational $\delta>0$ and returns $n$ such that $\varphi_n(x)$ has the same arity as $\varphi$ and $\|\varphi-\varphi_n\|<\delta$, the distance being the usual logical distance between $L$-formulae.  Moreover, if $\varphi$ is quantifier-free, then so are the $\varphi_n$.
\end{lem}

Given an $L$-structure $M$, a nonnegative $L$-formula $\varphi(x)$ is called an \textbf{almost-near formula for $M$} if, for any $\epsilon>0$, there is $\delta=\delta(\epsilon)>0$ so that, for any $a\in M$, if $\varphi^M(a)<\delta(\epsilon)$, then there is $b\in M$ such that $\varphi^M(b)=0$ and $d(a,b)\leq \epsilon$.  In this case, we refer to the function $\delta(\epsilon)$ as a \textbf{modulus} for $\varphi$.  If $\varphi$ is an almost-near formula for $M$, then we refer to the zeroset of $\varphi^M$ in $M$, denoted $Z(\varphi^M)$, as the \textbf{definable set} corresponding to $\varphi$. 

The utility of definable sets is that one can quantify over them in a first-order way.  In order to explain explicitly how we use this fact, we note that, given an almost-near formula $\varphi(x)$, \cite[Remark 2.12]{mtfms} establishes the existence of a nondecreasing, continuous function $\alpha:[0,\infty)\to [0,\infty)$ with $\alpha(0)=0$ and with the property that, for any $a\in M$, we have $d(a,Z(\varphi^M))\leq \alpha(\varphi^M(a))$; moreover, the function $\alpha$ depends only on the modulus $\delta(\epsilon)$ for $\varphi$.  As shown in the proof of \cite[Proposition 9.19]{mtfms}, it follows that $$d(a,Z(\varphi^M))=(\inf_{x}(\alpha(\varphi(x))+d(a,x)))^M. \quad \quad (\dagger)$$

The import of $(\dagger)$ is that the formula on the right-hand side of $(\dagger)$ is an actual formula of continuous logic.  We note that the proof appearing in \cite[Remark 2.12]{mtfms} shows that if the modulus $\delta$ is computable (when restricted to rational $\epsilon$), then the corresponding $\alpha$ is also a computable function.  We summarize this discussion as follows:

\begin{prop}\label{approx}
There is an algorithm such that takes as inputs a computable almost-near formula $\varphi(x)$ for $M$ that has a computable modulus and a rational $\eta>0$ and returns $n\in \bb N$ so that, for all $a\in M$, we have $|d(a,Z(\varphi^M))-\varphi_n(a)^M|<\eta$.  Moreover, if $\varphi$ is quantifier-free, then each $\varphi_n$ is existential.
\end{prop}

We note also that if $\varphi$ is an almost-near formula for $M$, then it is also an almost-near formula for any ultrapower $M^\u$ of $M$ (this is where the asymmetry in the types of inequalities used in the definition for almost-near formulae comes into play) and that the formula $(\dagger)$ and the conclusion of the previous proposition also hold for $M^\u$ as well.

In the remainder of this section, we discuss the notion of computability and decidability of theories.  An issue arises in that there are two common definitions of the theory of a metric structure.  While equivalent for model-theoretic purposes, the presence of these two different formulations creates some subtleties when bringing comutability-theoretic ideas into the picture.

First, given an $L$-structure $M$, the \textbf{theory of $M$} is the function $\Th(M)$ whose domain is the set of $L$-sentences and which is defined by $\Th(M)(\sigma):=\sigma^M$.  The \textbf{universal theory of $M$}, denoted $\Th_\forall(M)$, is the restriction of $\Th(M)$ to the set of universal $L$-sentences.

\begin{defn}
Let $M$ be an $L$-structure.  We say that \textbf{the (universal) theory of $M$ is computable} if there is an algorithm which takes as inputs a restricted (universal) $L$-sentence $\sigma$ and a rational number $\delta>0$ and returns rational numbers $a<b$ with $b-a<\delta$ and for which $\sigma^M\in (a,b)$.
\end{defn}

One also uses the word theory in continuous logic as a synonym for a set of $L$-sentences.  In this case, given an $L$-structure $M$, the theory of $M$ is the set $\{\sigma \ : \ \sigma^M=0\}$ and the universal theory of $M$ is the intersection of the theory of $M$ with the set of universal $L$-sentences.  Since a theory is a set of sentences, we believe the following terminology is appropriate:

\begin{defn}
A theory $T$ is \textbf{decidable} if there is an algorithm which, upon input a restricted $L$-sentence $\sigma$, decides whether or not $\sigma$ belongs to $T$.  Similarly, $T$ is \textbf{effectively enumerable} if there is an algorithm which enumerates the restricted $L$-sentences that belong to $T$.
\end{defn}

It is clear that each version of the theory of $M$ can be recovered from the other version, whence, from the point of view of model theory, there is no harm in blurring the distinction.  However, from the computability-theoretic perspective, there is a difference.  Indeed, while it is clear that the decidability of the theory of $M$ implies its computability, the converse need not be true.

There is a proof system for continuous logic, first introduced in \cite{BYP}.  There, one defines the relation $T\vdash \sigma$, where $T$ is a restricted $L$-theory and $\sigma$ is a restricted $L$-sentence.  A feature of this proof system is that, if $T$ is effectively enumerable, then so is the set of $\sigma$ such that $T\vdash \sigma$.  The following version of the completeness theorem, first proven in \cite{BYP}, will play a large role in the sequel:

\begin{fact}
For any restricted $L$-theory $T$ and any restricted $L$-sentnce $\sigma$, we have
$$\sup\{\sigma^M \ : \ M\models T\}=\inf\{r\in \bb Q^{>0} \ : \ T\vdash \sigma\dminus r\}.$$
\end{fact}

Suppose, in the previous display, that $\sigma$ is a universal sentence and that the common value is $0$.  If $T$ is effectively enumerable and we begin to enumerate the theorems of $T$, then we may never see the fact that $T\vdash \sigma$ even though $T\vdash \sigma\dminus \frac{1}{2^n}$ for all $n$.  This motivates the following definition:

\begin{defn}
Given an $L$-structure $M$, we say that the universal theory of $M$ is \textbf{weakly effectively enumerable} if one can effectively enumerate the sentences $\sigma\dminus r$, where $\sigma$ is a restricted universal sentence, $r\in \bb Q^{>0}$, and $\sigma^M\leq r$.
\end{defn}

For some structures (such as $\R$), the computability of the universal theory of the structure is equivalent to it being weakly effectively enumerable:

\begin{prop}
Suppose that $M$ is a separable $L$-structure that has a \textbf{computable presentation}.  Then $\Th_\forall(M)$ is computable if and only if it is weakly effectively enumerable.
\end{prop}

Roughly speaking, $M$ has a computable presentation if there is a countable, dense subset of $M$ so that one can uniformly approximately compute the values of the symbols in the langauge on the countable dense set.  As mentioned in \cite{GH} (and elaborated on in \cite{GH2}), $\R$ has a computable presentation.  The proof of the nontrivial direction of the previous proposition follows by using the computable presentation to perform a brute force lower bound approximation to the value of any universal sentence.

\section{Proof of the main theorem}

\begin{defn}

Fix $n,m\in \bb N$.

\begin{enumerate}
\item A sequence of projections $(C_a : a \leq m)$ such that $\sum_a C_a = 1$ is called a \textbf{projection valued measure} (PVM).  
\item The set $\cqs(n,m)$ of \textbf{quantum correlations} consists of the correlations of the form $p(a,b|x,y)= \langle A^x_a \otimes B^y_b\xi,\xi \rangle$ for $x,y \leq n$ and $a,b \leq m$,
where H is a finite-dimensional Hilbert space, $\xi \in H \otimes H$ is a unit vector, and for every $x, y \leq n$,
$(A^x_a: a\leq m)$ and $(B^y_b:b\leq m)$ are PVMs on H.
\item We set $\cqa(n,m)$ to be the closure in $[0,1]^{n^2k^2}$ of $\cqs(n,m)$.
\item Given an element $p\in C_{qa}(n,m)$, we say that $p$ is \textbf{synchronous} if\\ $p(i,j|v,v)=0$ whenever $i\not=j$.  We let $C_{qa}^s(n,m)$ denote the set of \textbf{synchronous} correlation matrices.  
\end{enumerate}
\end{defn}

\begin{defn}
A \textbf{nonlocal game $\mathfrak G$ with $n$ questions and $m$ answers} is a probability distribution $\mu$ on $n \times n$ together with a decision function 
\[
D:n\times n \times m \times m \rightarrow \{0,1\}.
\]
We call the nonlocal game \textbf{synchronous} if $D(v,v,i,j)=0$ whenever $i\not=j$.
\end{defn}

\begin{defn}
For each nonlocal game $\mathfrak G$, recall that the \textbf{entangled value} of $\mathfrak G$ is the quantity
$$\val^*(\mathfrak G)=\sup_{p\in C_{qa}(n,m)}\sum_{v,w}\mu(v,w)\sum_{i,j}D(v,w,i,j)p(i,j|v,w).$$ We also define the \textbf{synchronous value} of $\mathfrak G$ to be the quantity $$\pval^*(\mathfrak G)=\sup_{p\in C_{qa}^s(n,m)}\sum_{v,w}\mu(v,w)\sum_{i,j}D(v,w,i,j)p(i,j|v,w).$$
\end{defn}

In general, $\pval^*(\mathfrak G)\leq \val^*(\mathfrak G)$.  




The following is the main result of \cite{MIP*}:
\begin{thm}\label{main-MIP*}
There is an effective map $\cal M\mapsto \mathfrak G_\cal M$ from Turing machines to synchronous nonlocal games such that:
\begin{itemize}
    \item if $\cal M$ halts, then $\pval^*(\mathfrak G_\cal M)=1$;
    \item if $\cal M$ does not halt, then $\val^*(\mathfrak G_\cal M)\leq \frac{1}{2}$.
\end{itemize}
\end{thm}

\begin{remark}
The fact that the games  in the statement of the previous theorem are synchronous does not appear explicitly in \cite{MIP*} but is an artifact of their proof.  Moreover, in the case that $\cal M$ halts, the fact that a winning strategy can be taken to be synchronous is also an artifact of the proof (see Remark 5.12 in \cite{MIP*}).
\end{remark}



We let $\varphi_{n,m}(x_{v,i})$ ($i=1,\ldots,m$, $v=1,\ldots,n$) denote the (computable!) formula
$$\max\left(\max_{v,i}\|x_{v,i}^2-x_{v_i}\|_2,\max_{v,i}\|x_{v,i}^*-x_{v,i}\|_2,\max_v\|\sum_ix_{v,i}-1\|_2\right).$$
We let $X^N_{n,m}$ denote the zeroset of $\varphi_{n,m}$ in $N$.  Note that elements of $X^N_{n,m}$ are $n$-tuples of PVMs in $N$, where each PVM in the tuple consists of $m$ orthogonal projections.

\begin{thm}
Each formula $\varphi_{n,m}$ is an almost-near formula for $\mathcal R$ with a computable modulus. 
\end{thm}

\begin{proof}
This follows immediately from \cite[Lemma 3.5]{KPS} and its proof.
\end{proof}


Given a nonlocal game $\mathfrak G$, let $\psi_\mathfrak G(x_{v,i})$ denote the formula
$$\sum_{v,w}\mu(v,w)\sum_{i,j}D(v,w,i,j)\tr(x_{v,i}x_{w,j}).$$

\begin{thm}
For any game $\mathfrak G$, we have $$\pval^*(\mathfrak G)=\left(\sup_{x_{v,i}\in X_{n,m}}\psi_\mathfrak G(x_{v,i})\right)^{\R}.$$
\end{thm}

\begin{proof}
This follows immediately from the equivalence of (1) and (4) in \cite[Theorem 3.6]{KPS}.
\end{proof}

\begin{thm}\label{main}
Suppose that $\Th_\forall(\R)$ is computable.  Then for any computable game $\mathfrak G$ (meaning that the $\mu(v,w)$ are computable reals), we have that $\pval^*(\mathfrak G)$ is a computable real, uniformly in the description of $\mathfrak G$.
\end{thm}

\begin{proof}
For simplicity, set $X:=X_{n,m}$ and $x:=(x_{v,i})$.  Set $\sigma:=\sup_{x\in X}\psi_\mathfrak G(x_{v,i})$.  Note first that, since $\psi_\mathfrak G$ is 1-Lipshitz, we have $\sigma^\cal R=(\sup_x(\psi_\mathfrak G(x)\dminus d(x,X)))^\cal R$.  Now given $\eta>0$, one can effectively find an existential restricted formula $\varphi_n$ such that, for all $x\in \R^\u$, we have $|d(x,X^{\R^\u})-\varphi_n(x)^{\R^\u}|<\eta$.  It follows that $$\left| \sigma^\R-(\sup_x(\psi_\mathfrak G(x)\dminus \varphi_n(x)))^\R\right|<\eta.$$  Since the latter formula in the above display is equivalent to a universal restricted formula, the computability of the universal theory of $\R$ allows us to compute it to within $\eta$, and thus we can compute $\sigma^\R$ to within $2\eta$.  By the previous theorem, this is equivalent to being able to compute $\pval^*(\mathfrak G)$ to within $2\eta$.  

It is clear that these are considerations are uniform in the description of $\mathfrak G$.



\end{proof}

\begin{cor}
$\Th_\forall(\R)$ is not computable.
\end{cor}

\begin{proof}
Suppose, towards a contradiction, that $\Th_\forall(\cal R)$ is computable.  Given a Turing machine $\cal M$, we use the effective map from Theorem \ref{main-MIP*} to construct the computable game $\mathfrak G_M$.  Using the previous theorem, we can compute an interval $(a,b):=(a_\cal M,b_\cal M)\subseteq [0,1]$ of radius smaller than $\frac{1}{4}$ such that $\pval^*(\mathfrak G_\cal M)\in (a,b)$.  If $a>\frac{1}{2}$, then $\val^*(\mathfrak G_\cal M)\geq \pval^*(\mathfrak G_\cal M)>\frac{1}{2}$, whence $\val^*(\mathfrak G_\cal M)=1$ and $\cal M$ halts.  If $a\leq \frac{1}{2}$, then $b<\frac{3}{4}$, whence $\pval^*(\mathfrak G_\cal M)<\frac{3}{4}$.  Since $\mathfrak G_\cal M$ is special, we have that $\val^*(\mathfrak G_\cal M)<1$ and hence $\cal M$ does not halt.  Since this allows us to decide the halting problem, we have reached a contradiction.
\end{proof}

\section{A reformulation in terms of noncommutative moments}

In this section, we offer a reformulation of our main theorem in terms that might be more appealing to operator algebraists.

Given positive integers $n$ and $d$, we fix variables $x_1,\ldots,x_n$ and enumerate all *-monomials in the variables $x_1,\ldots,x_n$ of total degree at most $d$, $m_1,\ldots,m_L$. (Of course, $L=L(n,d)$ depends on both $n$ and $d$.)  We consider the map $\mu_{n,d}:\R_1^n \rightarrow D^L$ given by $\mu_{n,d}(\vec a)=(\tau(m_i(\vec a)) \ : \ i=1,\ldots,L)$.  (Here, $D$ is the complex unit disk.)


We let $X(n,d)$ denote the range of $\mu_{n,d}$ and $X(n,d,p)$ be the image of the unit ball of $M_p(\mathbb C)$ under $\mu_{n,d}$.  Notice that $\bigcup_{p\in \bb N} X(n,d,p)$ is dense in $X(n,d)$.

\begin{thm}
The following statements are equivalent:
\begin{enumerate}
    \item The universal theory of $\R$ is computable.
    \item There is a computable function $F:\bb N^3\to \bb N$ such that, for every $n,d,k\in \bb N$, $X(n,d,F(n,d,k))$ is $\frac{1}{k}$-dense in $X(n,d)$.
\end{enumerate}
\end{thm}

\begin{proof}
First suppose that the universal theory of $\R$ is computable.  We produce a computable function $F$ as in (2).  Fix $n$, $d$, and $k$, and set $\epsilon:=\frac{1}{3k}$.  Computably find $s_1,\ldots,s_t$, an $\epsilon$-net in $D^L$.  For each $i=1,\ldots,t$, ask the universal theory of $\R$ to compute intervals $(a_i,b_i)$ with $b_i-a_i<\epsilon$ and with $\left(\inf_{\vec x}|\mu_{n,d}(\vec x)-s_i|\right)^\R\in (a_i,b_i)$.  For each $i=1,\ldots,t$ such that $b_i<2\epsilon$, let $p_i\in \bb N$ be the minimal $p$ such that when you ask the universal theory of $\mathbb M_p(\bb C)$ to compute intervals of shrinking radius containing $\left(\inf_{\vec x}|\mu_{n,d}(\vec x)-s_i|\right)^{M_p(\bb C)}$, there is a computation that returns an interval $(c_i,d_i)$ with $d_i<2\epsilon$.  Let $p$ be the maximum of these $p_i$'s.  We claim that setting $F(n,d,k):=p$ is as desired.  Indeed, suppose that $s\in X(n,d)$ and take $i=1,\ldots,t$ such that $|s-s_i|<\epsilon$.  Then $\left(\inf_{\vec x}|\mu_{n,d} (\vec x)-s_i|\right)^\R<\epsilon$, whence $b_i<2\epsilon$.  It follows that there is an interval $(c_i,d_i)$ as above with $\left(\inf_{\vec x}|\mu_{n,d}(\vec x)-s_i|\right)^{M_p(\bb C)}<d_i<2\epsilon$.  Let $a\in M_p(\bb C)$ realize the infimum.  Then $|\mu_{n,d}(\vec a)-s|<3\epsilon=\frac{1}{k}$, as desired.

Now suppose that $F$ is as in (2).  We show that that the universal theory of $\R$ is computable.  Towards this end, fix a restricted universal sentence 
\[
\sigma = \sup_{\vec x} f(\tau(m_1),\ldots,\tau(m_\ell))
\]
where $\vec x = x_1,\ldots,x_n$ and $m_1,\ldots,m_\ell$ are *-monomials in $\vec x$ of total degree at most $d$.  Fix also rational $\epsilon>0$.  We show how to compute the value of $\sigma^\R$ to within $\epsilon$.  Since $f$ is a restricted connective, it has a computable modulus of continuity $\delta$.  Consequently, we can find $k\in \bb N$ computably so that $\frac{1}{k} \leq \delta(\epsilon)$.  Set $p = F(n,d,2k)$.  Computably construct a sequence $\vec a_1,\ldots,\vec a_t\in (M_p(\mathbb{C})_1)^n$ that is a $\frac{1}{2k}$ cover of $(M_p(\mathbb C)_1)^n$ (with respect to the $\ell^1$ metric corresponding to the 2-norm).  Consequently, $\mu_{n,d}(\vec a_1),\ldots,\mu_{n,d}(\vec a_t)$ is a $\frac{1}{2k}$-cover of $X(n,d,p)$.  Set 
\[
r:=\max_{i=1,\ldots,t}f(\tau(m_1(\vec a_i)),\ldots,\tau(m_l(\vec a_i))).
\]
By assumption, $X(n,d,p)$ is $\frac{1}{2k}$-dense in $X(n,d)$.  It follows that $r\leq \sigma^\R\leq r+\epsilon$, as desired.
\end{proof}



\begin{remark}
Notice that the *-monomials used in the proof of Theorem \ref{main} are of very low degree (at most 4) and so we have the stronger result that there is no computable function of the form $F(n,4,k)$ in the theorem above.
\end{remark}

\section{A general perspective on Embedding Problems}

Recall that a structure $N$ embeds into an ultrapower of another structure $M$ (in the same language) if and only if $N$ is a model of the universal theory of $M$.  All of the embedding problems in operator algebras attempt to find a small subset of the universal theory of some canonical object so that modeling that small subset suffices to conclude that one models the entire universal theory.  For example, the Connes Embedding Problem asks whether or not modeling the theory of tracial von Neumann algebras (which is a subset of the universal theory of $\cal R$) is enough to know that one models the entire universal theory of $\cal R$.  Similarly, the MF Problem asks whether or not modeling the theory of stably finite C*-algebras (which, again, is part of the universal theory of $\cal Q$) is enough to know that one models the entire universal theory of $\cal Q$.

Now that we know that the Connes Embedding Problem is false, it is natural to ask whether or not one can ``reasonably'' enlarge the theory of tracial von Neumann algebras in such a way that then modeling that enlarged theory does indeed imply that you model the entire universal theory of $\cal R$.  We show that, under one interpretation of ``reasonable,'' this is impossible.  We first offer the following general definition:

\begin{defn}
Given a structure $M$ in a language $L$, we call the \textbf{$M$EP} the statement that there is an effectively enumerable subset $T$ of the \emph{full} theory of $M$ such that, for any $L$-structure $N$, if $N\models T$, then $N$ embeds into an ultrapower of $M$.
\end{defn}

Note that this definition allows the possibility that the extra information being allowed need not be universal information, but rather can have arbitrary quantifier-complexity.  On the other hand, the restriction that $T$ be effectively enumerable is somewhat severe (although natural from the logical point of view).  


We have the following general statement:

\begin{thm}\label{EP}
If the $M$EP has a positive solution, then the universal theory of $M$ is weakly effectively enumerable.
\end{thm}

\begin{proof}
Suppose that there is a effectively enumerable subset $T$ of the theory of $M$ such that whenever $N\models T$, then $N$ embeds into an ultrapower of $M$.  It follows that, for any universal sentence $\sigma$, we have, using the Completeness Theorem, that
$$\sigma^M=\sup\{\sigma^N \ : \ N\models T\}=\inf\{r\in \mathbb Q^{>0} \ : \ T\vdash \sigma\dminus r\}.$$ The result now follows.
\end{proof}


Recalling that weak effective enumerability is equivalent to computability for the universal theory of $\R$, we now have the following strengthening of the fact that CEP has a negative solution:

\begin{cor}\label{REP}
$\cal R$EP has a negative solution.
\end{cor}

\begin{remark}
In the case of the $\R$EP, we can make an even stronger statement, namely that there is no effectively enumerable theory $T$ extending the theory of II$_1$ factors with the property that every model of $T$ embeds into an ultrapower of $\R$.  Note that we are not requiring that $\R$ itself be a model of $T$, but instead require that every model of $T$ be a II$_1$ factor.  Indeed, since every II$_1$ factor contains a copy of $\R$, the proof of Theorem \ref{EP} goes through and we obtain this stronger statement.
\end{remark}




As stated in the introduction, Corollary \ref{REP} allows us to provide ``many'' counterexamples to CEP:

\begin{cor}
There is a sequence $M_1,M_2,\ldots,$ of separable II$_1$ factors, none of which embed into an ultrapower of $\cal R$, and such that, for all $i<j$, $M_i$ does not embed into an ultrapower of $M_j$.
\end{cor}

\begin{proof}
We construct the sequence inductively.  Set $M_1$ to be any separable II$_1$ factor that does not embed into an ultrapower of $\cal R$.  Suppose now that $M_1,\ldots,M_n$ have been constructed satisfying the conclusion of the Corollary.  For each $i=1,\ldots,n$, let $\sigma_i$ be a nonnegative restricted sentence such that $\sigma_i^{\cal R}=0$ but $\sigma_i^{M_i}>0$.  For each $i=1,\ldots,n$, fix a rational number $\delta_i\in (0,\sigma_i^{M_i})$.  Let $T$ be the theory of II$_1$ factors together with the single condition $\max_{i=1,\ldots,n}(\sigma_i\dminus \delta_i)=0$.  It is clear that $T$ is a recursively enumerable subset of the theory of $\cal R$.  Thus, by Corollary \ref{REP}, there is a separable model $M_{n+1}$ of $T$ such that $M_{n+1}$ does not embed into an ultrapower of $\cal R$.  Given $i=1,\ldots,n$, since $\sigma_i^{M_i}>\delta_i$ while $\sigma_i^{M_{n+1}}\leq \delta_i$, it follows that $M_i$ does not embed into an ultrapower of $M_{n+1}$.  This indicates how to continue the recursive construction, completing the proof.
\end{proof}

Recall that a tracial von Neumann algebra $\cal S$ is called \textbf{locally universal} if every tracial von Neumann algebra embeds into an ultrapower of $\cal S$.  As shown in \cite{mtoa3}, there is a locally universal tracial von Neumann algebra and it is clear that they all have the same universal theory.  Since the theory of tracial von Neumann algebras is recursively axiomatizable, we arrive at the following result:

\begin{thm}
If $\cal S$ is a locally universal tracial von Neumann algebra, then the $\cal S$EP has a positive solution.
\end{thm}  

The same remark can be made for locally universal C*-algebras.

\section{An application to C*-algebras}

Recall that the \textbf{MF problem}, first posed by Blackadar and Kirchberg, asks whether or not every stably finite \cstar-algebra embeds into an ultrapower of the universal UHF algebra $\cal Q$.  The following consequence of the failure of CEP was pointed out to us by Thomas Sinclair and Aaron Tikuisis:

\begin{prop}
The MF problem has a negative solution.
\end{prop}

\begin{proof}
Suppose that $M$ is a II$_1$ factor that does not embed into $\R^\u$ (here we mean the tracial ultrapower).  We claim then that $M$ does not embed (as a C*-algebra) into a nonprincipal C*-ultrapower of $\cal Q$.  Indeed, suppose, towards a contradiction, that $i:M\hookrightarrow \cal Q^\u$ is such an embedding.  
Let $\pi:\mathcal Q^\u \to \mathcal R^\u$ denote the composition of the quotient map $\cal Q^\u\to \cal Q^\u/I$, where $I$ is the trace ideal, with the natural inclusion $\cal Q^\u/I\hookrightarrow \R^\u$.  Since $M$ has a unique trace, which is faithful, we get that the composition  $\pi\circ i:M\to \mathcal R^{\u}$ is an embedding and a *-homomorphism, yielding a contradiction.
\end{proof}

\begin{remark}
The counterexample to the MF problem in the proof above is not separable.  One can easily obtain a separable counterexample, e.g. by taking a separable elementary substructure in the language of \cstar-algebras.
\end{remark}

In this section, we improve upon this result by showing that the $\cal Q$EP has a negative solution.  This result will follow from a more general result applying to a wider class of C*-algebras.

\begin{defn}
Given $m\in \bb N$ and $0<\gamma<1$, we say that a unital \cstar-algebra $A$ has the \textbf{$(m,\gamma)$-uniform Dixmier property} if, for all self-adjoint $a\in A$, there are unitaries $u_1,\ldots,u_m\in U(A)$ and $z\in Z(A)$ such that
$$\left\|\sum_{i=1}^m \frac{1}{m}u_iau_i^*-z\right\|\leq \gamma\|a\|.$$ We say that $A$ has the \textbf{uniform Dixmier property} if it has the $(m,\gamma)$-Dixmier property for some $m$ and $\gamma$.  
\end{defn}

Clearly if $A$ has the $(m,\gamma)$-Dixmier property, then it has the $(m,\gamma')$-Dixmier property for any $0<\gamma<\gamma'<1$, whence we may always assume that $\gamma$ is dyadic rational.

Given $m$ and $\gamma$, let $\theta_{m,\gamma}$ denote the following sentence in the language of \cstar-algebras:
$$\sup_a\inf_{u_1,\ldots,u_n}\inf_\lambda\max\left(\max_{i=1,\ldots,n}\|u_iu_i^*-1\|,\|\sum_{i=1}^m\frac{1}{m}u_iau_i^*-\lambda\|\dminus \gamma\|a\|\right).$$
Here, the supremum is over self-adjoint contractions, the first infimum is over contractions, and the second infimum is over the unit disk in $\bb C$.  If $A$ is a simple unital \cstar-algebra with the $(m,\gamma)$-uniform Dixmier property, then $\theta_{m,\gamma}^A=0$.  Conversely, if $\theta_{\gamma,m}^B=0$, then $B$ is monotracial.

Given a tracial C*-algebra $(A,\tau_A)$, one lets $N_{(A,\tau_A)}$ denote the weak closure of $A$ in the GNS representation corresponding to $\tau_A$.  It is known that $N_{(A,\tau_A)}$ is isomorphic to the algebra obtained from taking the $\|\cdot\|_{2,\tau_A}$-completion of each bounded ball of $A$.  

We are now ready to prove our main theorem of this section:

\begin{thm}
Suppose that $A$ is an infinite-dimensional, unital, simple C*-algebra with the uniform Dixmier property and such that $N_{(A,\tau_A)}$ embeds into an ultrapower of $\cal R$.  Then the $A$EP has a negative solution.
\end{thm}

\begin{proof}
Suppose, towards a contradiction that the $A$EP has a positive solution as witnessed by the theory $T_0$.  Fix $m$ and $\gamma$ with $\gamma$ a dyadic rational such that $A$ has the $(m,\gamma)$-uniform Dixmier property.  Let $L$ be the language of tracial \cstar-algebras and let $T$ be the union of the following three $L$-theories:
\begin{itemize}
    \item $T_0$;
    \item the $L$-theory of tracial \cstar-algebras;
    \item the single condition $\theta_{m,\gamma}=0$.
\end{itemize}  
Note that $T$ is recursively axiomatizable.  Clearly $(A,\tau_A)\models T$.  Now suppose that $(M,\tau_M)\models T$.  Since $M\models T_0$, there is an embedding $M\hookrightarrow A^\u$.  Since $M\models \sigma_{m,\gamma}=0$, $M$ is monotracial, whence this embedding is trace preserving, that is, we have an embedding $(M,\tau_M)\hookrightarrow (A,\tau_A)^\u$.  Consequently, for any universal $L$-sentence $\sigma$, we have that $$\sigma^{(A,\tau_A)}=\sup\{\sigma^{(M,\tau_M)} \ : \ (M,\tau_M)\models T\}=\inf\{r\in \bb Q^{>0} \ : \ T\vdash \sigma\dminus r\},$$  where the second equality follows from the Completeness Theorem.  Thus, by running proofs from $T$, we obtain approximations from above to the value of $\sigma^{(A,\tau_A)}$.  If $\sigma$ is a sentence in the language of tracial von Neumann algebras, then $\sigma$ can be construed in the language of tracial \cstar-algebras.  In this case, set $N:=N_{(A,\tau_A)}$.  Since $A$ is monotracial, $N$ is a II$_1$ factor; call its trace $\tau_N$. Note that since $A$ is simple, $\tau_A$ is faithful and so $A$ embeds into $N$.  It follows that we have $\sigma^{(A,\tau_A)}=\sigma^{(N,\tau_N)}$.  Since $N$ is a II$_1$ factor that embeds into an ultrapower of $\cal R$, we have that $\sigma^{(N,\tau_N)}=\sigma^{(\cal R,\tau_\cal R)}$.  We thus have that the universal theory of $\R$ is weakly effectively enumerable, which is a contradiction.
\end{proof}

We remind the reader of a theorem of Haagerup and Zsidó \cite{hz}, namely that a simple unital \cstar-algebra has the Dixmier property if and only if it is monotracial.  In particular, $\cal Q$ and $\cal Z$ have the Dixmier property.  Relevant for our discussion is the following:

\begin{fact}
$\cal Q$ and $\cal Z$ have the uniform Dixmier property.
\end{fact}

\begin{proof}
\cite[Corollary 3.11]{art} states that all unital AF \cstar-algebras with the Dixmier property have the uniform Dixmier property, whence $\cal Q$ has the uniform Dixmier property.
\cite[Remark 3.18 and Corollary 3.22]{art} shows that $\cal Z$ has the uniform Dixmier property.
\end{proof}

\begin{cor}
The $\cal Q$EP and $\cal Z$EP have negative solutions.
\end{cor}

\begin{remark}
A specific consequence of the previous corollary is that there is a unital, stably projectionless C*-algebra that does not embed into $\cal Z^\u$.  Here, a C*-algebra $A$ is said to be stably projectionless if, for every $n$, the only projections in $M_n(A)$ are, up to Murray-von Neumann equivalence, 0 or of the form $p \otimes 1$, where $p$ is a projection in $M_n(\mathbb{C})$.  (Note that this is indeed expressible by a set of $\forall\exists$ axioms true of $\cal Z$.)  It would be interesting to see if one could derive this conclusion from the failure of CEP alone using purely operator algebra techniques.
\end{remark}




\section{Tsirelson's problem and Kirchberg's QWEP Conjecture revisited}

As discussed in the introduction, the result MIP*=RE implies a negative solution to Tsirelson's problem.  This conclusion is achieved by applying a semidefinite programming argument coupled with a noncommutative Positivstellensatz result.  In this section, we show how to replace this latter argument with a simple argument using the Completeness Theorem and a result from \cite{quantum}, whose proof is essentially just an application of the Cauchy-Schwarz inequality.

We first remind the reader of the definitions relevant to state Tsirelson's problem.

\begin{defn}
The set $\cqc(n,m)$ of \textbf{quantum commuting correlations} consists of the correlations of the form $p(a,b|x,y)= \langle A^x_a B^y_b\xi,\xi \rangle$ for $x,y \leq n$ and $a,b \leq m$,
where H is a separable Hilbert space, $\xi \in H$ is a unit vector, and for every $x, y \leq n$,
$(A^x_a: a\leq m)$ and $(B^y_b:b\leq m)$ are PVMs on H for which $A^x_a B^y_b=B^y_bA^x_a$.  As before, if $p(a,b|x,x)=0$ whenever $a\not=b$, we call the quantum correlation synchronous and we let $\cqc^s(n,m)$ denote the set of synchronous quantum commuting correlations.  If $\mathfrak G$ is a nonlocal game, we set
$$\pval^{co}(\mathfrak G)=\sup_{p\in C_{qc}^s(n,m)}\sum_{v,w}\mu(v,w)\sum_{i,j}D(v,w,i,j)p(i,j|v,w).$$
\end{defn}

Tsirelson's problem asks whether or not $C_{qa}(n,m)=C_{qc}(n,m)$ for all $n$ and $m$.  Our derivation of a negative solution to Tsirelson's problem from MIP*=RE also establishes a complexity-theoretic fact, which we now state.  (To be fair, this result is also derivable from the aforementioned semidefinite programming/Positivstellensatz argument.)

\begin{defn}
Fix $0<r\leq 1$.  We define $\mip^{co,s}_{0,r}$ to be the set of those languages $L$ (in the sense of complexity theory) for which there is an efficient mapping $z\mapsto \mathfrak G_z$ from strings to nonlocal games such that $z\in L$ if and only if $\pval^{co}(\mathfrak G_z)\geq r$.
\end{defn}

The terminology $\mip^{co}_0$ already exists in the literature and denotes the set of languages for which there exists an efficient mapping $z\mapsto \mathfrak G_z$ as above such that $z\in L$ if and only if $\val^{co}(\mathfrak G_z)=1$, where $\val^{co}$ of a game is defined as $\pval^{co}$ except one takes the supremum over $\cqc(n,m)$ instead of just over $\cqc^s(n,m)$.  Our aim is to prove the following:

\begin{thm}\label{mipcos}
For any $0<r\leq 1$, every language in $\mip^{co,s}_{0,r}$ belongs to the complexity class coRE.
\end{thm}

In other words, if $L\in \mip^{co,s}_{0,r}$, then there is an algorithm which enumerates the complement of $L$.

The following is \cite[Corollary 5.6]{quantum}
\begin{fact}\label{vern}
The correlation $p(i,j|v,w)$ belongs to $C_{qc}^s(n,k)$ if and only if there is a C*-algebra $A$, a tracial state $\tau$ on $A$, and a generating family of projections $p_{v,i}$ such that $\sum_{i=1}^k p_{v,i}=1$ for each $v=1,\ldots,n$ and such that $p(i,j|v,w)=\tau(p_{v,i}p_{w,j})$.
\end{fact}

Recall the formula $\psi_\mathfrak G(x_{v,i})$ from Section 3.  Let $\theta_{\mathfrak G,r}$ be the sentence 
\[
\inf_{x_{v,i}}\max\left(\max_{v,i}(\|x_{v,i}^2-x_{v,i}\|,\max_{v,i}\|x_{v_i}^*-x_{v,i}\|,\max_v\|\sum_i x_{v,i}-1\|,r\dotminus \psi_\mathfrak G(x_{v,i})\right).
\]

Let $T$ be the theory of tracial C*-algebras as in the previous section.  The following is immediate from Fact \ref{vern}:

\begin{prop}\label{vern2}
For any nonlocal game $\mathfrak G$, we have $\pval^{co}(\mathfrak G)\geq r$ if and only if the theory $T\cup\{\theta_{\mathfrak G,r}=0\}$ is satisfiable.
\end{prop}

We will also need the following immediate consequence of the 
Completeness Theorem:

\begin{lem}\label{con}
Let $U$ be a continuous theory.  Then $U$ is satisfiable if and only if $U\not\vdash \bot$\footnote{$\bot$ represents a contradiction i.e. any continuous sentence which cannot evaluate to 0.  For instance, the constant function 1.}.
\end{lem}

We can now prove Theorem \ref{mipcos}.  Let $L$ belong to $\mip^{co,s}_{0,r}$.  Given a string $z$, let $\mathfrak G_z$ be the corresponding game.  If $z\notin L$, then Proposition \ref{vern2} and Lemma \ref{con} imply that $T\cup\{\theta_{\mathfrak G_z,r}=0\} \vdash \bot.$
 Since this latter condition is recursively enumerable, the proof of Theorem \ref{mipcos} is complete.  


One can now deduce the failure of Tsirelson's problem from MIP*=RE as follows.  Suppose, towards a contradiction, that $C_{qa}^s(n,m)=C_{qc}^s(n,m)$ for every $n$ and $m$.  Let $\cal M\mapsto \mathfrak G_\cal M$ be the efficient mapping from Turing machines to nonlocal games described in Theorem \ref{main-MIP*}.  Given a Turing machine $\cal M$, one simultaneously starts computing lower bounds on $\pval^*(\mathfrak G_\cal M)$ while running proofs from $T\cup\{\theta_{\mathfrak G_{\cal M,1}}=0\}$.  Either the first computation eventually yields the fact that $\pval^*(\mathfrak G_{\cal M})>\frac{1}{2}$, in which case $\cal M$ halts, or else the second computation eventually yields the fact that $T\cup\{\theta_{\mathfrak G_{\cal M,1}}=0\}\vdash \bot$,
in which case $\pval^*(\mathfrak G_\cal M)<1$, and $\cal M$ does not halt.  In this way, we can decide the halting problem, a contradiction.  Note that we derived the a priori stronger statement that $\cqa^s(n,m)\not=\cqc^s(n,m)$ for some $n$ and $m$.  

Although somewhat implicit in \cite{KPS}, we can now quickly derive a negative solution to Kirchberg's QWEP conjecture.  Indeed, based on the previous paragraph, it suffices to show that if the QWEP conjecture had a positive solution, then $C_{qa}^s(n,m)=C_{qc}^s(n,m)$ for all $n$ and $m$.  Towards this end, fix $p\in C_{qc}^s(n,m)$ and take a tracial C*-algebra $(A,\tau)$ generated by projections $p_{v,i}$ as in Fact \ref{vern}.  Recall from \cite{KPS} that $C^*(\bb F(n,m))$ is the universal C*-algebra generated by projections as in Fact \ref{vern}.  Letting $e_{v,i}$ denote the corresponding projections in $C^*(\bb F(n,m))$, we fix a surjective *-homomorphism $\pi:C^*(\bb F(n,m))\to A$ sending $e_{v,i}$ to $p_{v,i}$.  Let $\tau'$ be the trace on $C^*(\bb F(n,m))$ defined by $\tau'(a):=\tau(\pi(a))$.  Since $C^*(\bb F(n,m))$ has the local lifting property, if the QWEP conjecture were true, it would also have the weak expectation property, whence $\tau'$ would be an amenable trace.  (See \cite{BO} for all the of the terms and facts described in the previous sentence.)  By the equivalence of (1) and (3) in \cite[Theorem 3.6]{KPS}, it follows that $p\in C_{qa}^s(n,m)$, as desired.

\end{document}